\documentclass[11pt,a4paper,reqno]{amsart}
\usepackage{amsmath,amsthm,amsfonts,amssymb,amscd,mathrsfs}
\usepackage[backref,colorlinks,linkcolor=blue!60!green!80!red,dvips]{hyperref}
\usepackage{array} 
\usepackage{xfrac} 
\usepackage{tikz}

\def\vvs{\vskip 1cm}
\def\vs{\vskip 0.30cm}
\def\vss{\vskip 0.05cm}
\def\la{\lambda}
\def\ep{\varepsilon}
\def\n1{\medmuskip=0mu n-1}
\def\k1{\medmuskip=0mu k+1}
\def\ktwo{\medmuskip=0mu k+2}
\def\leqs{\leqslant}
\def\geqs{\geqslant}

\newcommand{\Z}{\mathbb{Z}}

\theoremstyle{plain}
\newtheorem{thm}{Theorem}
\newtheorem{lem}{Lemma}

\newtheorem{cor}{Corollary}
\newtheorem*{conj}{Conjecture}
\newtheorem{qn}{Question}

\newcommand{\set}[1]{\left\{ #1 \right\}}

\title{On a conjecture of {H.\ Gupta}}
\author[Lecouturier]{Emmanuel Lecouturier$^*$}
\email{lecouturier.emmanuel@gmail.com}

\author[Zmiaikou]{David Zmiaikou$^{**}$}
\email{david.zmiaikou@gmail.com}

\date{\today}
\thanks{
$^*$Lyc\'ee Blaise Pascal, 20 rue Alexander Fleming, 91400 Orsay, France.
\\
\indent $^{**}$D\'epartement de Math\'ematiques, Universit\'e Paris-Sud 11, 91405 Orsay Cedex, France}
 
\begin{document}
\maketitle
\begin{abstract}
Denote by $r(n)$ the length of a shortest integer sequence on a circle containing all permutations of the set $\{1,2,...,n\}$ as subsequences. Hansraj Gupta conjectured in $1981$ that $r(n) \leqs \frac{n^2}{2}$. In this paper we confirm the conjecture for the case where $n$ is even, and show that $r(n) < \frac{n^2}{2} + \frac{n}{4} -1$ if $n$ is odd.
\end{abstract}

\vvs

\setcounter{section}{-1}
\section{Introduction}
Let $n>1$ be a positive integer. One would like to find the minimal number $r=r(n)$ for which there exists a sequence $x_1$, $x_2$, \ldots , $x_r$ such that any permutation of the first $n$ natural numbers is a subsequence of $x_j$, $x_{j+1}$, \ldots , $x_r$, $x_1$, $x_2$, \ldots , $x_{j-1}$ for at least one $j$, $1 \leqs j \leqs r$.

In other words, if we write the sequence $x_1$, $x_2$, \ldots , $x_r$ on the circumference of a circle, then any permutation of the set $\{1,2,...,n\}$ can be located starting from a point on the circle, moving in the clockwise direction and never turning back nor crossing the point of start. Such a sequence is called \textit{a rosary} of degree $n$. Examples of rosaries of degree $2$, $3$, $4$ and $5$ are given in Figure \ref{fig:rosary1}.

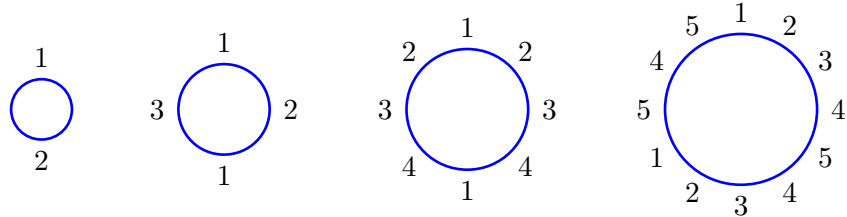
\begin{figure}[htbp]
   \begin{center}
   \begin{tikzpicture}[scale=.4]
      \draw[blue,line width=1pt] (0,0) circle (1cm);
      \foreach \angle/\num in {90/1,270/2}
      { \node at (\angle:1.7cm) {\num}; }
      
      \begin{scope}[xshift=6cm]
      \draw[blue,line width=1pt] (0,0) circle (1.5cm);
      \foreach \angle/\num in {0/2,90/1,180/3,270/1}
      { \node at (\angle:2.2cm) {\num}; }
      \end{scope}

      \begin{scope}[xshift=14cm]
      \draw[blue,line width=1pt] (0,0) circle (2cm);
      \foreach \angle/\num in {0/3,45/2,90/1,135/2,180/3,225/4,270/1,315/4}
      { \node at (\angle:2.7cm) {\num}; }
      \end{scope}

      \begin{scope}[xshift=23cm]
      \draw[blue,line width=1pt] (0,0) circle (2.5cm);
      \foreach \angle/\num in {0/4,30/3,60/2,90/1,120/5,150/4,180/5,210/1,240/2,270/3,300/4,330/5}
      { \node at (\angle:3.2cm) {\num}; }
      \end{scope}
   \end{tikzpicture}
   \caption{Examples of rosaries.}
   \label{fig:rosary1}
   \end{center}
\end{figure}

It is easy to show that $r(n) \leqs n^2-3n+4$, since the following sequence 
$$1 (2,\ldots ,n)_{n-2}\, 2 = 1,\underbrace{2, \ldots , n, 2, \ldots , n , \ldots, 2, \ldots , n}_{(n-2) \textrm{ times}}, 2$$
is a rosary of degree $n$. Indeed, for an arbitrary permutation $a_1=1$, $a_2$, \ldots , $a_n$ of the numbers $1$, $2$, \ldots , $n$, either the string $a_2$, \ldots , $a_n$ is a concatenation of at most $(n-2)$ increasing strings, or $a_2=n > a_3=n-1 > \ldots > a_n = 2$. 

The following conjecture was suggested by Hansraj Gupta in \cite{gupta}. It can also be found in the well-known book \cite{guy} under Problem E22.

\begin{conj}[H.\ Gupta]
The inequality  $r(n) \leqs \frac{n^2}{2} $ is satisfied for all integers $n>1$.
\end{conj}

\vs Here is the structure of the present paper:
\begin{enumerate}
  \item[1)] For any even $n$, a rosary of length $\frac{n^2}{2}$ is constructed.
  \item[2)] Examples for which the strict inequality $r(n) < \frac{n^2}{2}$ takes place are given.
  \item[3)] For any odd $n$, we construct a rosary of length less than $\frac{n^2}{2}+\frac{n}{4}-1$, and give examples illustrating that the expected estimation $\frac{n^2}{2}$ is not reachable by our method. 
  \item[4)] The asymptotics $r(n) \sim \frac{n^2}{2}$ as $n \rightarrow +\infty$ is established.
\end{enumerate}

\vs
\noindent \textbf{Acknowledgements.} We would like to express our gratitude to Jean-Christophe Yoccoz, the Ph.D. thesis advisor of the second author, and to Tsimafei Khatkevich for useful discussions and comments. 

\vvs
\section{A proof for $n$ even}
By default all sequences we deal with here are cyclic (one may imagine that they are written on a circle). A cyclic sequence will be shortly called a \emph{cycle}. If we want to emphasize that $a_1, a_2, \dots, a_n$  is considered as an ordinary sequence (and not a cycle), then it will be called a \emph{string}.

A \textit{block} of a sequence (cycle or string) is a subsequence of its consecutive elements. For instance $(4,5,3)$ and $(2,1,7,4)$ are blocks of the cycle $7$, $4$, $5$, $3$, $6$, $2$, $1$. A \textit{maximal increasing block} is an increasing block $a_{j+1} < a_{j+2} < \ldots < a_{j+k}$ such that $a_j>a_{j+1}$ and $a_{j+k} > a_{j+k+1}$. By analogy we define a \emph{maximal decreasing block}. In the example above the blocks $(7, 4)$, $(5, 3)$ and $(6, 2, 1)$ are maximal decreasing, and the blocks $(4, 5)$, $(3, 6)$, $(2)$ and $(1, 7)$ are maximal increasing.

The \emph{code} of a cycle $a_1, a_2, \dots, a_n$ is the following cyclic sequence
\begin{equation*}
   H(a_2 - a_1),\; H(a_3 - a_2),\; \dots,\; H(a_n - a_{n-1}),\; H(a_1 - a_n),
\end{equation*}
where $H$ denotes the Heaviside step function, $H(t)=\left\{ \begin{array}{cl} 0 & \textrm{if } t<0,\\ 1 & \textrm{if } t\geqs 0.\end{array}\right.$
Thus, the code of the cycle $7$, $4$, $5$, $3$, $6$, $2$, $1$ is $0, 1, 0, 1, 0, 0, 1$.

The \emph{code} of a string $a_1, a_2, \dots, a_n$ is the following string
\begin{equation*}
   H(a_2 - a_1),\; H(a_3 - a_2),\; \dots,\; H(a_n - a_{n-1}).
\end{equation*}
Note that the code of a string of length $n$ is a string of length $(n-1)$.

\vs Consider a permutation $a_1, a_2, \dots, a_n$ of the first $n$ natural numbers regarded as a cyclic sequence. Let the cycle $c_1, c_2, \dots, c_n$ be its code. Suppose that the code has $x$ ones $c_{i_1}=c_{i_2}= \ldots= c_{i_x}=1$ and $y$ zeros, where $x+y=n$. 

\hspace{-0.9cm}
\tabcolsep=1em
\begin{tabular}{m{10.5cm}m{4cm}}
\parindent=1em Denote by $\la_1$ the number of zeros between $c_{i_1}$ and $c_{i_2}$, by $\la_2$ the number of zeros between $c_{i_2}$ and $c_{i_3}$, etc., by $\la_x$ the number of zeros between $c_{i_x}$ and $c_{i_1}$. We have
\begin{center}
   $\la_1 + \la_2 + \ldots + \la_x = y.$
\end{center}
It is clear that the integers $(\la_i+1)$ are lengths of the maximal decreasing blocks of the cycle $a_1, a_2, \dots, a_n$.
&
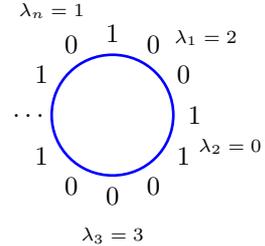
\begin{tabular}{m{4cm}}
\begin{tikzpicture}[scale=.4]
      \draw[blue,line width=1pt] (0,0) circle (2cm);
      \foreach \angle/\num in {0/1,30/0,60/0,90/1,120/0,150/1,180/$\dots$,210/1,240/0,270/0,300/0,330/1}
      { \node at (\angle:2.7cm) {\small\num}; }
      \foreach \angle/\num in {40/$\la_1=2$,-15/$\la_2=0$,-90/$\la_3=3$,120/$\la_n=1$}
      { \node at (\angle:4cm) {\tiny \num}; }
\end{tikzpicture}
\end{tabular}
\end{tabular}

\begin{lem} \label{lem:1}
Let $K$ and $M$ be positive integers. If there exists an index $i\in [1, x]$ such that
\begin{equation*}
  \la_{i+1} + \la_{i+2} + \ldots + \la_{i+K} \geqs y-M+1,
\end{equation*}
then the cycle $a_1, a_2, \dots, a_n$ is a subsequence of the cycle\footnote{We write $(x_1, x_2, \dots, x_m)_k$ for the juxtaposition of $k$ copies of the sequence in the brackets.}
$$(1, 2, \dots, n)_{M-1}\, (1, 2, \dots, \n1)\, (n, \dots, 2, 1)_{K-1}\, (n, \dots, 3, 2).$$
\end{lem}
\begin{proof}
Without loss of generality (up to a cyclic permutation), we may suppose that the code of the cycle $a_1, a_2, \dots, a_n$ has the following form
\begin{equation*}
   c_1, c_2, \dots, c_n = 1 (0)_{\la_1} 1 (0)_{\la_2} \dots 1 (0)_{\la_{x-K}} 1 (0)_{\la_{x-K+1}} \dots 1 (0)_{\la_x},
\end{equation*}
where
\begin{equation*}
  \la_{x-K+1} + \la_{x-K+2} + \ldots + \la_{x} \geqs y-M+1,
\end{equation*}
and so
\begin{equation*}
  \la_{1} + \la_{2} + \ldots + \la_{x-K} \leqs M-1.
\end{equation*}
A part $(0)_{\la} 1 = \underbrace{0, 0, \dots, 0}_{\la}, 1$ of the code corresponds to inequalities $a_{\bullet}> a_{\bullet+1}> \ldots > a_{\bullet+\la}<a_{\bullet+\la+1}$. On one hand, this gives $\la$ increasing blocks $(a_{\bullet+1}), \ldots, (a_{\bullet+\la-1}), (a_{\bullet+\la}<a_{\bullet+\la+1})$, and on the other hand, one decreasing block $a_{\bullet}> a_{\bullet+1}> \ldots > a_{\bullet+\la}$. This means that the sequence $a_1, a_2, \dots, a_n$ is a juxtaposition of 
\begin{equation*}
  1+\la_{1} + \la_{2} + \ldots + \la_{x-K} \leqs M
\end{equation*}
increasing blocks, the first one being $(a_1<a_2)$, and $K$ decreasing blocks encoded by 
$$(0)_{\la_{x-K+1}} 1 (0)_{\la_{x-K+2}} \dots 1 (0)_{\la_x}\,.$$
Hence, it is a subsequence of $(1, 2, \dots, n)_{M-1}\, (1, 2, \dots, \n1)\, (n, \dots, 2, 1)_{K-1}\, (n, \dots, 3, 2)$.
\end{proof}

\vs The following theorem shows that $r(n) \leqs \frac{n^2}{2}$ for any even positive integer $n$.

\setlength{\leftmargini}{2.5em}
\begin{thm} \label{th:1}
Let $k$ be a positive integer.
\begin{itemize}
   \item[\textbf{a)}] If $n=4 k$, then the sequence
   \begin{equation}\label{eq:rosary4K}
   (1, 2, \dots, n)_k\, (1, n, \n1, \dots, 2)_k 
   \end{equation}
   is a rosary of degree $n$.\vs

   \item[\textbf{b)}] If $n=4 k+2$, then the sequence
   \begin{equation}\label{eq:rosary4K2}
   (1, 2, \dots, n)_{k+1}\, (1, n, \n1, \dots, 2)_k
   \end{equation}
   is a rosary of degree $n$.
\end{itemize}
\end{thm}
\begin{proof} 
\textbf{a)} Assume that a permutation $a_1, a_2, \dots, a_n$ of the numbers $1, 2, \dots, n$ does not occur in (\ref{eq:rosary4K}). Then by Lemma \ref{lem:1} with $K=M=k$, one has
\begin{center}
  \tikz{\node[draw,rounded corners,text centered] at (0,0) {$\la_{i+1} + \la_{i+2} + \ldots + \la_{i+k} \leqs y-k, \quad\textrm{for any } 1\leqs i\leqs x.$};}
\end{center}
Summing up the $x$ inequalities leads to
\begin{multline*}
  (\la_{1} + \la_{2} + \ldots + \la_{k}) + (\la_{2} + \la_{3} + \ldots + \la_{k+1}) + \ldots + (\la_{x} + \la_{1} + \ldots + \la_{k-1}) \\
 = k(\la_{1} + \la_{2} + \ldots + \la_{x}) = k y \leqs x(y-k). 
\end{multline*}
Making the substitutions $y=n-x$ and $n=4 k$, one obtains
\begin{eqnarray*}
k(n-x)\leqs x(-x+n-k), \\
x^2-n x+k n\leqs 0, \\
(x-2 k)^2\leqs 0. 
\end{eqnarray*}
Hence, \underline{$x= 2 k$, $y=2 k$}, and all the inequalities are actually equalities. In particular,  
\begin{equation}\label{eq:1}
  \la_{i+1} + \la_{i+2} + \ldots + \la_{i+k} = y-k=k. \qquad \textrm{for any } 1\leqs i\leqs x.
\end{equation}
Let us show that in this case the cycle $a_1, a_2, \dots, a_n$ is a subsequence of (\ref{eq:rosary4K}). Without loss of generality, we may take $a_1=1$ implying also that  $c_n=0$ and $c_1=1$ (as $a_n> a_1< a_2$). 

First, the code of the string
$$a_1=1, a_2, a_3, \dots, a_{2k+1}$$
equals the part $1 (0)_{\la_1} 1 (0)_{\la_2} \dots 1 (0)_{\la_{k}}$ of the sequence $c_1, c_2, \dots, c_n$, and occurs in the string
$$(1, n, \n1, \dots, 2)_k,$$
since $1 (0)_{\la}$ encodes inequalities $a_{\bullet}< a_{\bullet+1}> a_{\bullet+2}> \ldots > a_{\bullet+\la}>a_{\bullet+\la+1}$, and thus the string $a_2, a_3, \dots, a_{2k+1}$ consists of $k$ maximal decreasing blocks.

Second, the remaining string
$$a_{2k+2}, a_{2k+3}, \dots, a_{4k}$$
has the code $(0)_{\la_{k+1}} 1 (0)_{\la_{k+2}} \dots 1 (0)_{\la_{2k}\,-1}$ and is a subsequence of the following string
$$(1, 2, \dots, n)_k.$$
Indeed, each part $1 (0)_{\la_i}$ gives $\la_i$ increasing blocks $(a_{\bullet}< a_{\bullet+1}), (a_{\bullet+2}), \ldots, (a_{\bullet+\la_i})$, so that the last term $a_{\bullet+\la_i+1}$ is included in the first increasing block $(a_{\bullet+\la_i+1}< a_{\bullet+\la_i+2})$ for the part $1 (0)_{\la_{i+1}}$. There is an exception for $1 (0)_{\la_{2k}\,-1}$: the last term $a_{4k}$ has to be counted as a separate increasing block. Hence, the string $a_{2k+2}, a_{2k+3}, \dots, a_{4k}$ consists of 
$$ \la_{k+1} + \la_{k+2} + \ldots + (\la_{2k}-1)+1 =k$$
increasing blocks due to (\ref{eq:1}).

\vss
\begin{small}
\noindent\hspace{1cm}
$\begin{array}{rcl}
  \textrm{An example for } k=2, n=8: \quad a_1, a_2, \dots, a_8 &=& 1, 5, 7, 6, 3, 4, 8, 2 \\
  c_1, c_2, \dots, c_8 &=& 1, 1, 0, 0, 1, 1, 0, 0 \\
  \la_1, \la_2, \la_3, \la_4 &=& 0, 2, 0, 2,
\end{array}$

\hspace{0.7cm}
where we get two decreasing blocks $(5)$, $(7, 6, 3)$ and two increasing blocks $(4, 8)$, $(2)$.
\end{small}

\vss
Therefore, the string $a_1, a_2, \dots, a_n$ is a subsequence of the string
$$(1, n, \n1, \dots, 2)_k\, (1, 2, \dots, n)_k.$$
This brings a contradiction the initial assumption and, thus, proves the point \textbf{a)} of the theorem.

\vs\noindent\textbf{b)} For $n=4k+2$, we proceed on the analogy of \textbf{a)}. Assume that a cycle $a_1, a_2, \dots, a_n$ is not a subsequence of the cycle
$$(1, 2, \dots, n)_{k+1}\, (1, n, \n1, \dots, 2)_k.$$
Then by Lemma \ref{lem:1} with $K=k$ and $M=k+1$, it follows that
\begin{center}
  \tikz{\node[draw,rounded corners,text centered] at (0,0) {$\la_{i+1} + \la_{i+2} + \ldots + \la_{i+k} \leqs y-k-1, \quad\textrm{for any } 1\leqs i\leqs x.$};}
\end{center}
Summing up all these $x$ inequalities, one gets
$$k(\la_{1} + \la_{2} + \ldots + \la_{x}) = k y \leqs x(y-k-1).$$
The substitutions $y=n-x$ and $n=4k+2$ give
\begin{eqnarray*}
k(n-x)\leqs x(-x+n-k-1), \\
x^2 - (n-1) x+k n\leqs 0, \\
x^2 - (4k+1) x+k (4k+2)\leqs 0, \\
(x-2 k)(x-(2k+1))\leqs 0. 
\end{eqnarray*}
Hence, the integer $x$ must lie in the interval $[2k, 2k+1]$, that is, either $x=2 k$ or $x=2k+1$.

\vss
If \underline{$x=2k$, $y=2k+2$}, then 
\begin{equation*}
  \la_{i+1} + \la_{i+2} + \ldots + \la_{i+k} = y-k-1=k+1, \qquad \textrm{for any } 1\leqs i\leqs x.
\end{equation*}
By analogy with \textbf{a)}, the string $a_1=1, a_2, \dots, a_{2k+2}$ has the code $1 (0)_{\la_1} 1 (0)_{\la_2} \dots 1 (0)_{\la_{k}}$ and is a subsequence of the string 
$$(1, n, \n1, \dots, 2)_k.$$
The code of the string $a_{2k+3}, a_{2k+4}, \dots, a_{4k+2}$ equals $(0)_{\la_{k+1}} 1 (0)_{\la_{k+2}} \dots 1 (0)_{\la_{2k}\,-1}$, where $\la_{k+1} + \la_{k+2} + \ldots + \la_{2k} =k+1$, and so the string occurs as a subsequence in the string $$(1, 2, \dots, n)_{k+1}.$$

\vss
If \underline{$x=2k+1$, $y=2k+1$}, then the following equalities take place 
\begin{equation*}
  \la_{i+1} + \la_{i+2} + \ldots + \la_{i+k} = y-k-1=k, \qquad \textrm{for any } 1\leqs i\leqs x,
\end{equation*}
and so $\la_{i+k} - \la_i = (\la_{i+1}+\ldots+\la_{i+k}) - (\la_{i}+\ldots+\la_{i+k-1})=k-k=0$. Since the integers $k$ and $(2k+1)$ are coprime, all the numbers $\la_{1}, \la_{2}, \ldots, \la_{2k+1}$ are equal (to 1). This means that the code of the cycle $a_1, a_2, \dots, a_n$ is $(1, 0)_{2k+1}=1, 0, 1, 0, \dots, 1, 0$. Taking without loss of generality $a_1=1$, we have $k$ decreasing blocks
$(a_2>a_3),\; (a_4>a_5),\; \dots,\; (a_{2k}>a_{2k+1})$
and $(k+1)$ increasing blocks
$(a_{2k+2}),\; (a_{2k+3}<a_{2k+4}),\; \dots,\; (a_{4k+1}<a_{4k+2}).$
Therefore, the string $a_1=1, a_2, \dots, a_n$ is a subsequence of 
$(1, n, \n1, \dots, 2)_k\, (1, 2, \dots, n)_k.$

\vss
\begin{small}
\noindent\hspace{1cm}
$\begin{array}{rcl}
  \textrm{An example for } k=2, n=10: \quad a_1, a_2, \dots, a_{10} &=& 1, 4, 2, 10, 6, 7, 3, 9, 5, 8 \\
  c_1, c_2, \dots, c_{10} &=& 1, 0, 1, 0, 1, 0, 1, 0, 1, 0 \\
  \la_1, \la_2, \la_3, \la_4, \la_5 &=& 1, 1, 1, 1, 1,
\end{array}$

\hspace{0.7cm}
where we get the blocks $(4>2)$, $(10>6)$ and $(7)$, $(3<9)$, $(5<8)$, so that the 

\hspace{0.7cm}
string $\{a_i\}$ occurs in $(1, 10, 9, \dots, 2)_2\, (1, 2, \dots, 10)_3$.
\end{small}

\vss
Drawing a conclusion, in both cases we have a contradiction to the assumption that $a_1, a_2, \dots, a_n$ is not a subsequence of (\ref{eq:rosary4K2}).
\end{proof}

\vs
\section{Examples where $r(n)< \frac{n^2}{2}$ }
We are going to show that the cycle given in Figure \ref{fig:rosary2} is a rosary of degree 6 and length 17. Indeed, let $a_1, a_2, a_3, a_4, a_5, a_6$ be a permutation of the first $6$ positive integers. 
\begin{figure}[htbp]
   \begin{center}
   \begin{tikzpicture}[scale=.5]
      \draw[blue,line width=1pt] (0,0) circle (2.5cm);
      \foreach \angle/\num in {0/$1_1$,-21/2,-42/3,-63/4,-84/5,-105/6,-126/$1_2$,-147/3,-168/4,-189/5,-210/6,-231/2,-252/$1_3$,-273/6,-294/5,-315/4,-336/3}
      { \node at (\angle:3.2cm) {\num}; }
   \end{tikzpicture}
   \caption{A rosary of degree 6.}
   \label{fig:rosary2}
   \end{center}
\end{figure}
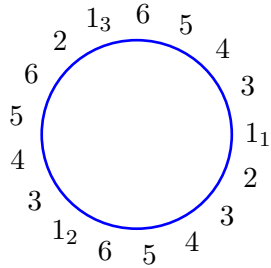
We would like to pick the permutation among the numbers on the circle, moving in the clockwise direction. It is sufficient to check the 11 cases that are listed in the Table \ref{tab:counterexample}.

\begin{table}[htdp]
\begin{center} 
\caption{Any permutation of $1, 2, \dots, 6$ is a subsequence of the cycle in Figure \ref{fig:rosary2}.}
\begin{tabular}{| c | l | c |}
  \hline
  \textbf{The permutation} & \hspace{1.2cm}\textbf{Cases} & \begin{tabular}{l}\textbf{The number on the} \\ \textbf{circle to start with} \end{tabular}\\ \hline\hline
  $1, 2, a_3, a_4, a_5, a_6$ & $a_3<a_4 \textrm{ or } a_5>a_6$ & $1_1$ \\ 
        & $a_3>a_4 \textrm{ and } a_5<a_6$ & $1_2$ \\ \hline\hline
        
  $1, a_2, 2, a_4, a_5, a_6$ & $a_4>a_5$ & $1_2$ \\
        & $a_4<a_5$ & $1_3$ \\ \hline\hline
        
  $1, a_2, a_3, 2, a_5, a_6$ & $a_2<a_3$ & $1_2$ \\
        & $a_2>a_3$ & $1_3$ \\ \hline\hline
        
  $1, a_2, a_3, a_4, 2, a_6$ & $a_2<a_3 \textrm{ or } a_3<a_4$ & $1_1$ \\
        & $a_2>a_3>a_4$ & $1_3$ \\ \hline\hline
        
  $1, a_2, a_3, a_4, a_5, 2$ & $a_2<a_3 \textrm{ and } a_4<a_5$ & $1_1$ \\
        & $a_2<a_3 \textrm{ and } a_4>a_5$ & $1_2$ \\
        & $a_2>a_3$ & $1_ 3$ \\ \hline
\end{tabular}
 \label{tab:counterexample}
\end{center}
\end{table} 

Here are some other rosaries of degree 6 and length 17:
$$\begin{array}{c}
1, 2, 3, 4, 5, 6, 1, 2, 1, 5, 6, 4, 3, 1, 2, 6, 5; \\
1, 2, 3, 4, 5, 6, 1, 2, 3, 5, 6, 4, 3, 2, 1, 6, 5; \\
1, 2, 3, 4, 5, 6, 1, 2, 4, 5, 6, 3, 1, 2, 6, 5, 4; \\
1, 2, 3, 4, 5, 6, 1, 3, 2, 5, 4, 3, 2, 1, 6, 4, 5; \\
1, 2, 3, 4, 5, 6, 1, 3, 4, 2, 1, 6, 5, 4, 3, 5, 6; \\
1, 2, 3, 4, 5, 6, 1, 3, 5, 6, 4, 2, 1, 6, 5, 4, 3; \\
1, 2, 3, 4, 5, 6, 1, 3, 6, 5, 4, 2, 1, 6, 5, 4, 3; \\
1, 2, 3, 4, 5, 6, 1, 4, 3, 2, 1, 6, 5, 2, 3, 4, 6; \\
1, 2, 3, 4, 5, 6, 1, 4, 5, 3, 2, 1, 6, 2, 3, 5, 4; \\
1, 2, 3, 4, 5, 6, 1, 5, 4, 3, 6, 2, 1, 6, 3, 4, 5.\,
\end{array}$$
A long list of such rosaries, containing a block $(1, 2, 3, 4, 5, 6)$, was found by a program written by the authors in Java. 

An example for $n=8$: the cycle
$$1, 2, 3, 4, 5, 6, 7, 8,  1, 2, 3, 4, 5, 6, 7, 8, 6, 5, 4, 3, 2, 1, 7, 6, 5, 4, 3, 2, 1, 8, 7$$
is a rosary of length 31.

\begin{qn} 
Do there exist infinitely many \emph{even} positive integers $n$ for which $r(n)< \frac{n^2}{2}$?
\end{qn}

\vvs
\section{An upper bound in the case that $n$ is odd}
Consider a permutation $a_1, a_2, \dots, a_n$ of the first $n$ natural numbers regarded as a cyclic sequence. Let the cycle 
$$c_1, c_2, \dots, c_n = 1 (0)_{\la_1} 1 (0)_{\la_2} \dots  1 (0)_{\la_x}$$ 
be its code consisting of $x$ ones and $y$ zeros, that is,
$$\la_1 + \la_2 + \ldots + \la_x = y \quad \textrm{and}\quad x+y=n.$$ 

A part $1 (0)_{\la_i}$ of the cycle $c_1, c_2, \dots, c_n$ encodes inequalities $$a_{\bullet}< a_{\bullet+1}> a_{\bullet+2}> \ldots > a_{\bullet+\la_i}>a_{\bullet+\la_i+1}.$$ An index $i\in [1, x]$ will be called \emph{$(K, M, N)$-lucky} if 
$$\la_{i+1} + \la_{i+2} + \ldots + \la_{i+K}\geqs y-M$$
and the number $a_{\bullet+\la_i+1}$ does not exceed $N$. 

\begin{lem}\label{lem:2}
If there exists a $(K, M, N)$-lucky index $i$, then the permutation $a_1, a_2, \dots, a_n$ is a subsequence of the cycle
$$(1, 2, \dots, n)_{M}\, (1, 2, \dots, N)\, (n, \dots, 2, 1)_{K}.$$
\end{lem}
\begin{proof}[Proof of Lemma \ref{lem:2}] 
We have $\la_{i+1} + \la_{i+2} + \ldots + \la_{i+K} \geqs y-M$ and $a_{\bullet+\la_i+1}\leqs N$. Then the cycle $a_1, a_2, \dots, a_n$ is a juxtaposition of 

\renewcommand{\labelitemi}{$\triangleright$}
\begin{itemize}
   \item $K$ decreasing blocks $\left(a_{\bullet+1}> a_{\bullet+2}> \ldots > a_{\bullet+\la_j}>a_{\bullet+\la_j+1}\right)$, where $j \in [i+1,\; i+K]$, which correspond to
   $$(0)_{\la_{i+1}}1(0)_{\la_{i+2}}\ldots 1(0)_{\la_{i+K}},$$
   
   \item at most $M$ increasing blocks $(a_{\bullet}<a_{\bullet+1})$, $(a_{\bullet+2})$, $\ldots$, $(a_{\bullet+\la_j})$, where $j \in [i+K+1,\; x]\,\cup\, [1,\; i]$, which correspond to
   $$1(0)_{\la_{i+K+1}}\ldots 1(0)_{\la_x}\,1(0)_{\la_1}\ldots 1(0)_{\la_{i}\, - 1}\, ,$$
   for the reason that $$\sum_{j=i+K+1}^x \la_j + \sum_{j=1}^{i} \la_j = y - \sum_{j=i+1}^{i+K} \la_j\leqs M,$$
   
   \item one increasing block $(a_{\bullet+\la_i+1})$, where $a_{\bullet+\la_i+1}\leqs N$.
\end{itemize}
Hence, the cycle $(1, 2, \dots, n)_{M}\, (1, 2, \dots, N)\, (n, \dots, 2, 1)_{K}$ contains the permutation $a_1$, $a_2$, $\dots$, $a_n$ as a subsequence.
\end{proof}

The following theorem provides rosaries of odd degrees, its proof is a refinement of the method for Theorem \ref{th:1}.
\setlength{\leftmargini}{2.5em}
\begin{thm}\label{th:2}
Let $k$ be a positive integer.
\begin{itemize}
   \item[\textbf{a)}] If $n=4 k+1$, then the sequence
   \begin{equation}\label{eq:rosary4K1}
   (1, 2, \dots, n)_{k}\, (1, 2, \dots, 3k)\, (n, \dots, 2, 1)_{k-1}\, (n, \dots, 2)
   \end{equation}
   is a rosary of degree $n$.\vs

   \item[\textbf{b)}] If $n=4 k+3$, then the sequence
   \begin{equation}\label{eq:rosary4K3}
   (1, 2, \dots, n)_{k}\, (1, 2, \dots, 3\ktwo)\, (n, \dots, 2, 1)_{k}\, (n, \dots, 2)
   \end{equation}
   is a rosary of degree $n$.
\end{itemize}
\end{thm}
\begin{proof} 
\textbf{a)} Let us assume that the cycle $a_1, a_2, \dots, a_n$ is not a subsequence of (\ref{eq:rosary4K1}). Applying Lemma \ref{lem:1} with $K=M=k$, we obtain
\begin{equation*}
\la_{i+1} + \la_{i+2} + \ldots + \la_{i+k} \leqs y-k, \qquad\textrm{for each } 1\leqs i\leqs x.
\end{equation*}
Denote by $E\subset\set{1, \dots, x}$ the set of indices $i$ such that the following equality holds
$$\la_{i+1} + \la_{i+2} + \ldots + \la_{i+k} = y-k,$$
and so for all $i\notin E$, we have the inequality
$$\la_{i+1} + \la_{i+2} + \ldots + \la_{i+k} \leqs y-k-1.$$
Let $q$ be the cardinal of the set $E$. Summing up the $q$ equalities and $(x-q)$ inequalities gives
\begin{multline*}
  k y = (\la_{1} + \la_{2} + \ldots + \la_{k}) + (\la_{2} + \la_{3} + \ldots + \la_{k+1}) + \ldots + (\la_{x} + \la_{1} + \ldots + \la_{k-1}) \\
  \leqs q\cdot (y-k) + (x-q)\cdot (y-k-1)= x (y-k) - (x-q). 
\end{multline*}
Since  $y=n-x$, we get
\begin{eqnarray*}
k(n-x)\leqs x(-x+n-k) - (x-q), \\
x^2-(n-1) x+k n-q\leqs 0. 
\end{eqnarray*}
Therefore, the discriminant of the polynomial $t^2-(n-1) t+k n-q$ must be nonnegative:
$$\mathscr{D} = (n-1)^2 - 4(k n - q) = (4k)^2 - 4k(4k+1) + 4q = 4q-4k\geqs 0,$$
that is, $q\geqs k$. 
We have $q$ indices $i\in E$ for which
$$\la_{i+1} + \la_{i+2} + \ldots + \la_{i+k} = y-k.$$

Consider the set of the last terms $a_{\bullet+\la_i+1}$ of the maximal decreasing blocks $$a_{\bullet+1}> a_{\bullet+2}> \ldots > a_{\bullet+\la_i}>a_{\bullet+\la_i+1}$$ encoded by the parts $(0)_{\la_i}$ of $c_1, c_2, \ldots, c_n$ where $i\in E$. 
The numbers $a_{\bullet+\la_i+1}$ are different and don't exceed $n=4k+1$. Moreover, such a number cannot be equal to $n$, else the block $a_{\bullet+1}>  \ldots >a_{\bullet+\la_i+1}$ wouldn't be maximal. 

We shall now investigate the cases where $q>k$ and $q=k$ separately.

\vss\vss
\underline{Case $q>k$}.
Since the interval $[3\k1, 4k]$ contains $k$ integers and $|E|=q\geqs k+1$, there exists an index $e\in E$ such that $a_{\bullet+\la_e+1}\leqs 3k$. Then the index $e$ is $(k, k, 3k)$-lucky, and so according to Lemma \ref{lem:2}, the permutation $a_1, a_2, \dots, a_n$ is a subsequence of the cycle (\ref{eq:rosary4K1}).

\vss\vss
\underline{Case $q=k$}. The inequality $x^2-(n-1) x+k n-q=(x-2k)^2\leqs 0$ implies that $x=2k$, and so $y=2k+1$. For exactly $k=|E|$ values of $i\in \set{1, 2, \dots, 2k}$, we have the equality 
\begin{equation*}
   \la_{i+1} + \la_{i+2} + \ldots + \la_{i+k} = k+1,
\end{equation*}
For others, the equality $ \la_{i+1} + \la_{i+2} + \ldots + \la_{i+k} = k$ takes place, since $\la_1+\la_2+\ldots+\la_{2k}=2k+1$. Suppose that $a_{\bullet+\la_i+1} > 3k$ for all indices $i\in E$. Counting the number of integers in the maximal decreasing blocks $a_{\bullet+1}>  \ldots >a_{\bullet+\la_i+1}$, we conclude that there are at least
\begin{equation}\label{eq:1}
   \sum_{i\in E} (\la_i + 1) = k+\sum_{i\in E} \la_i
\end{equation}
integers in the interval $[3\k1, 4\k1]$. Hence, we have $\sum_{i\in E} \la_i\leqs 1$. Let $u\in E$ be an index such that $\medmuskip=0mu u+1\notin E$ (modulo $2k$). The number $a_{\bullet+\la_u+2}$ is not counted in (\ref{eq:1}), and yet $3k<a_{\bullet+\la_u+1}<a_{\bullet+\la_u+2}$. If $\sum_{i\in E} \la_i= 1$, then the interval $[3\k1, 4\k1]$ contains at least
$$\left(k+\sum_{i\in E}\la_i\right) + 1= k+2$$
integers, which is impossible.

If $\sum_{i\in E} \la_i= 0$, then there are two distinct indices $u, v\in E$ such that $\medmuskip=0mu u+1, v+1\notin E$ (modulo $2k$), otherwise, the set $E$ would be a subset of $k$ consecutive elements of the cycle $1, 2, \dots, 2k$, and so
$\sum_{i\in E} \la_i\geqs k>0$. We thus get two numbers $a_{\bullet+\la_u+2}$ and $a_{\bullet+\la_v+2}$ for which $3k<a_{\bullet+\la_u+1}<a_{\bullet+\la_u+2}$ and $3k<a_{\bullet+\la_v+1}<a_{\bullet+\la_v+2}$. Together with (\ref{eq:1}), this gives 
$$\left(k+\sum_{i\in E}\la_i\right) + 2= k+2$$
distinct integers in the interval $[3\k1, 4\k1]$, which is impossible. Therefore, $a_{\bullet+\la_i+1} \leqs 3k$ for at least one index $i\in E$. Such an index is $(k, k, 3k)$-lucky, and according to Lemma \ref{lem:2}, the permutation $a_1, a_2, \dots, a_n$ is a subsequence of the cycle (\ref{eq:rosary4K1}).

\vss\vss
In both cases we obtain a contradiction to the initial assumption that $a_1, a_2, \dots, a_n$ is not a subsequence of the cycle (\ref{eq:rosary4K1}).

\vs\noindent\textbf{b)} Assume that the permutation $a_1, a_2, \dots, a_n$ is not a subsequence of the cycle (\ref{eq:rosary4K3}). According to Lemma \ref{lem:1} with $K=k+1$ and $M=k$, we have
\begin{equation*}
\la_{i+1} + \la_{i+2} + \ldots + \la_{i+k+1} \leqs y-k, \qquad\textrm{for each } 1\leqs i\leqs x.
\end{equation*}
Let $E\subset\set{1, \dots, x}$ be the set of indices $i$ such that
$$\la_{i+1} + \la_{i+2} + \ldots + \la_{i+k+1} = y-k.$$
Thus,  for any $i\notin E$, the following inequality holds
$$\la_{i+1} + \la_{i+2} + \ldots + \la_{i+k+1} \leqs y-k-1.$$
Summing up these $q=|E|$ equalities and $(x-q)$ inequalities gives
\begin{multline*}
  (k+1) y = (\la_{1} + \la_{2} + \ldots + \la_{k+1}) + (\la_{2} + \la_{3} + \ldots + \la_{k+2}) + \ldots + (\la_{x} + \la_{1} + \ldots + \la_{k}) \\
  \leqs q\cdot (y-k) + (x-q)\cdot (y-k-1)= x (y-k) - (x-q), 
\end{multline*}
Since  $y=n-x$, we get
\begin{eqnarray*}
(k+1)(n-x)\leqs x(-x+n-k) - (x-q), \\
x^2-n x+(k+1) n-q\leqs 0. 
\end{eqnarray*}
Therefore, the discriminant of the polynomial $t^2-n t+(k+1) n-q$ must be nonnegative:
$$\mathscr{D} = n^2 - 4(k+1) n +4q = (4k+3)^2 - 4(k+1)(4k+3) + 4q = 4q-(4k+3)\geqs 0,$$
that is, $q\geqs k+1$ as $q\in \Z$. 
We have $q$ indices $i\in E$ for which
$$\la_{i+1} + \la_{i+2} + \ldots + \la_{i+k+1} = y-k.$$

Consider the set of the last terms $a_{\bullet+\la_i+1}$ of the maximal decreasing blocks $$a_{\bullet+1}> a_{\bullet+2}> \ldots > a_{\bullet+\la_i}>a_{\bullet+\la_i+1}$$ encoded by the parts $(0)_{\la_i}$ of $c_1, c_2, \ldots, c_n$ where $i\in E$. 
The numbers $a_{\bullet+\la_i+1}$ are different and don't exceed $n=4k+3$. Moreover, such a number cannot be equal to $n$, else the block $a_{\bullet+1}> a_{\bullet+2}> \ldots > a_{\bullet+\la_i}>a_{\bullet+\la_i+1}$ wouldn't be maximal. 
Therefore, since $q\geqs k+1$, there exists an index $e\in E$ such that $a_{\bullet+\la_e+1}\leqs 3k+2$. Then the index $e$ is $(\k1, k, 3\ktwo)$-lucky, and so by Lemma \ref{lem:2}, the permutation $a_1, a_2, \dots, a_n$ is a subsequence of the cycle (\ref{eq:rosary4K3}).

This contradiction to the initial assumption proves the statement of the theorem.
\end{proof}

For $n=4k+1$, the length of the rosary (\ref{eq:rosary4K1}) of degree $n$ is equal to
$$2k\cdot n + 3k-1=8k^2+5k -1< \frac{16k^2+8k+1}{2}+\frac{4k+1}{4}-1= \frac{n^2}{2}+\frac{n}{4}-1.$$

For $n=4k+3$, the length of the rosary (\ref{eq:rosary4K3}) of degree $n$ is equal to
$$(2k+1)\cdot n + 3k+1=8k^2+13k +4 < \frac{16k^2+24k+9}{2}+\frac{4k+3}{4}-1= \frac{n^2}{2}+\frac{n}{4}-1.$$

\begin{cor} 
If $n>1$ is odd, then $r(n)<\frac{n^2}{2}+\frac{n}{4}-1$.
\end{cor}

\vs
Let us now bring examples showing that one shouldn't expect obtaining the inequality $r(n)\leqs \frac{n^2}{2}$ by improving the construction of Theorem \ref{th:2}. 

Consider the permutation given in Figure \ref{fig:permutation}.
\begin{figure}[htbp] \label{fig:permutation}
   \begin{center}
   \begin{tikzpicture}[scale=.6]
      \draw[step=1cm,help lines] (0,-2) grid (21,1);
      
      \foreach \x/\num in {0/1, 1/5, 2/4, 3/3, 4/18, 5/17, 6/19, 7/20, 8/21, 9/6, 10/9, 11/8, 12/7, 13/12, 14/11, 15/10, 16/13, 17/14, 18/15, 19/16, 20/2}
      { \draw (\x,0)+(.5,.5) node {\num}; }
      
      \foreach \x/\xx/\l in {0/2/2,2/3/1,3/5/2,5/9/4,9/11/2,11/12/1,12/14/2,14/15/1,15/20/5,20/21/1}
      {\path[-,red,thick] (\x,-1) edge node[xshift=3pt,yshift=-3pt]{\tiny \l} (\xx,0);}

      \foreach \x/\xx/\l in {1/4/3,4/6/2,6/7/1,7/8/1,8/10/2,10/13/3,13/16/3,16/17/1,17/18/1,18/19/1}
      {\path[-,blue,thick] (\x,-1) edge node[xshift=-3pt,yshift=-3pt]{\tiny \l} (\xx,-2);}
       \path[-,blue,thick] (19,-1) edge (21,-1.666)
       					(0,-1.666) edge (1,-2);
      \node[blue,thick,xshift=-3pt,yshift=-3pt] at (20.5,-1.5) {\tiny 3};
   \end{tikzpicture}
    \caption{A permutation of the numbers $1, 2, \ldots, 21$.}
   \end{center}
\end{figure}
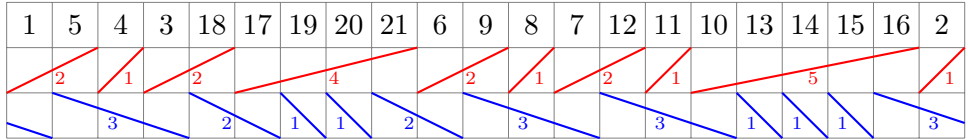
It has 10 maximal increasing blocks and 11 maximal decreasing blocks (they are illustrated by the ascending and the descending segments respectively). By trying to arrange 5 consecutive increasing blocks of the cycle into the string $(1, 2, \dots, 21)_5$ so that remaining part would fit in the string $(1, 2, \dots, 12)\, (21, \dots, 2, 1)_5$, one receives evidence that the permutation is not a subsequence of the cycle
$$(1, 2, \dots, 21)_5\, (1, 2, \dots, 12)\, (21, \dots, 2, 1)_4\, (21, \dots, 2)$$
of length $221= \left[\sfrac{21^2}{2}\right]+1$.
Thus, in the statement of Theorem \ref{th:2} with  $k=5$ and $n=21$, we couldn't have replaced the block $(1, 2, \dots, 3k)$ by the block $(1, 2, \ldots, 2\k1)$ in order to obtain the estimation $r(n)\leqs \frac{n^2}{2}$.

An example for $k=8$ and $n=33$: the permutation 
\begin{multline*} 
1,5,4,3,2,7,6,11,10,9,8,18,19,20,21,22,23,14,13,12,25,\\
24,26,17,16,15,28,27,29,30,31,32,33
\end{multline*} 
is not a subsequence of the cycle $(1, 2, \dots, 33)_8\, (1, 2, \dots, 17)\, (33, \dots, 2, 1)_7\, (33, \dots, 2)$ of length $544= \left[\sfrac{33^2}{2}\right]$.

\begin{qn} 
What is the value of $r(n)$ when $n=21$ or $33$?
\end{qn}

\vvs
\section{The asymptotic behaviour of $r(n)$}
According to H.\ Gupta, his conjecture was motivated by the problem of constructing a shortest \emph{string} which contains all permutations of the numbers $1, 2, \dots, n$ as subsequences. The latter was first suggested by R.\ M.\ Karp (as mentioned in \cite{knuth}).

Denote by $s(n)$ the length of such a string. The inequality $s(n)\leqs n^2 - 2n + 4$ was provided by several people though explicit constructions (see, for instance, \cite{adleman, koutas, mohanty}). 

In 1975, D. J. Kleitman and D. J. Kwiatkowski \cite{kleitman} established the lower bounds
$$n^2 - C_{\ep}\, n^{7/4+\ep}<s(n),$$
for any real number $\ep>0$ and a constant $C_{\ep}$ depending on $\ep$. 

\vs\vs
\begin{thm} 
The function $r$ is equivalent to $\frac{n^2}{2}$ as integer $n$ tends to $+\infty$.
\end{thm}
\begin{proof} 
By Theorems \ref{th:1} and \ref{th:2}, the inequality $r(n)<\frac{n^2}{2}+\frac{n}{4}$ is satisfied for any natural $n>1$. On the other hand, the function $r$ is bounded below:
$$s(n)\leqs 2 r(n).$$
Indeed, if $a_1, a_2, \dots, a_n$ is a rosary of degree $n$, then the string $a_1, a_2, \dots, a_n, a_1, a_2, \dots, a_n$ contains all permutations of the first $n$ natural numbers as subsequences. 

Finally, using the result of Kleitman and Kwiatkowski, we obtain 
$$\frac{n^2}{2} - C\, n^{15/8} < r(n) < \frac{n^2}{2}+\frac{n}{4}$$
for a constant $C>0$. Therefore, $r(n)\sim \frac{n^2}{2}$ as $n\to +\infty$.
\end{proof}

\vvs

\vvs

\end{document}